\documentclass[12pt]{amsart}
\usepackage{amsmath,amssymb,amsthm,amscd,amsfonts,geometry}

\theoremstyle{plain}
\newtheorem{thm}{Theorem}[section]
\newtheorem{prop}[thm]{Proposition}

\newtheorem{lem}[thm]{Lemma}

\newtheorem{main}{Main Theorem}

\theoremstyle{remark}

\theoremstyle{definition}

\newcommand{\I}{\mathbf{I}}                       

\newcommand{\cl}{\operatorname{cl}}               

\newcommand{\card}{\operatorname{card}}	          

\newcommand{\diam}{\operatorname{diam}}           

\newcommand{\fin}{\operatorname{Fin}}             
\newcommand{\cpt}{\operatorname{Comp}}             

\title[Topological structures of hyperspaces of finite sets]{Topological structures of hyperspaces of finite sets in non-separable metrizable spaces}

\author[K.~Koshino]{Katsuhisa Koshino}
\thanks{This work was supported by JSPS KAKENHI Grant Number 15K17530}

\address{Faculty of Engineering, Kanagawa University, Yokohama, 221-8686, Japan}

\email{pt120401we@kanagawa-u.ac.jp}

\subjclass[2010]{Primary: 54B20, Secondary: 54F65, 57N20}

\keywords{hyperspace, the Vietoris topology, the Hausdorff metric, Hilbert space, $Z$-set, Strong $Z$-set, strong universality, discrete cells property}

\begin{document}

\begin{abstract}
Let $\fin(X)$ be the hyperspace consisting of non-empty finite subsets of a space $X$ endowed with the Vietoris topology.
In this paper, we characterize a metrizable space $X$ whose hyperspace $\fin(X)$ is homeomorphic to the linear subspace spanned by the canonical orthonormal basis of a non-separable Hilbert space.
\end{abstract}

\maketitle

\section{Introduction}

Throughout this paper, spaces are metrizable, maps are continuous and $\kappa$ is an infinite cardinal.
By $\fin(X)$, we denote the hyperspace of non-empty finite sets in a space $X$ endowed with the Vietoris topology.
Let $\ell_2(\kappa)$ be the Hilbert space of density $\kappa$ and $\ell_2^f(\kappa)$ be the linear subspace spanned by the canonical orthonormal basis of $\ell_2(\kappa)$.
Topological structures of hyperspaces are classical subjects and have been studied in infinite-dimensional topology.
In 1985, D.W.~Curtis and N.T.~Nhu \cite{CN} characterized a space $X$ whose hyperspace $\fin(X)$ is homeomorphic to $\ell_2^f(\omega)$ as follows:

\begin{thm}\label{l2f}
A space $X$ is non-degenerate, connected, locally path-connected, strongly countable-dimensional\footnote{A space $X$ is strongly countable-dimensional if it is written as a countable union of finite-dimensional closed subsets.} and $\sigma$-compact\footnote{A space $X$ is called to be $\sigma$-(locally )compact if it is a countable union of (locally )compact subsets.} if and only if $\fin(X)$ is homeomorphic to $\ell_2^f(\omega)$.
\end{thm}

By an \textit{$X$-manifold}, we mean a topological manifold modeled on a space $X$.
In the general case, K.~Mine, K.~Sakai and M.~Yaguchi \cite{MSY} proved the following:

\begin{thm}
For a connected $\ell_2^f(\kappa)$-manifold $X$, the hyperspace $\fin(X)$ is homeomorphic to $\ell_2^f(\kappa)$.
\end{thm}

In this paper, we shall try to generalize their results as follows:

\begin{main}
A space $X$ is connected, locally path-connected, strongly countable-dimensional, $\sigma$-locally compact and for each point $x \in X$, any neighborhood of $x$ in $X$ is of density $\kappa$ if and only if $\fin(X)$ is homeomorphic to $\ell_2^f(\kappa)$.
\end{main}

\section{Preliminaries}

In this section, we will fix some notation and introduce a characterization of $\ell_2^f(\kappa)$-manifolds used in the proof of the main theorem.
Moreover, we show some lemmas concerning subdivisions of simplicial complexes that will be needed in the sequel.
The closed unit interval $[0,1]$ is denoted by $\I$.
Let $X = (X,d)$ be a metric space.
For a point $x \in X$ and subsets $A, B \subset X$, we put $d(x,A) = \inf\{d(x,a) \mid a \in A\}$ and $d(A,B) = \inf\{d(a,b) \mid a \in A, b \in B\}$.
For $\epsilon > 0$, we define $B_d(x,\epsilon) = \{x' \in X \mid d(x,x') < \epsilon\}$, $\overline{B_d}(x,\epsilon) = \{x' \in X \mid d(x,x') \leq \epsilon\}$, $N_d(A,\epsilon) = \{x' \in X \mid d(x',A) < \epsilon\}$ and $\overline{N_d}(A,\epsilon) = \{x' \in X \mid d(x',A) \leq \epsilon\}$.
By $\diam_d{A}$, we mean the diameter of $A$.
It is well-known that the topology of $\fin(X)$ coincides with the one induced by \textit{the Hausdorff metric} $d_H$ defined as follows:
 $$d_H(A,B) = \inf\{r > 0 \mid A \subset N_d(B,r), B \subset N_d(A,r)\}.$$

For maps $f : X \to Y$ and $g : X \to Y$, and for an open cover $\mathcal{U}$ of $Y$, $f$ is \textit{$\mathcal{U}$-close} to $g$ if for each point $x \in X$, there is a member $U \in \mathcal{U}$ that contains the both $f(x)$ and $g(x)$.
A closed subset $A$ of a space $X$ is called to be a \textit{(strong) $Z$-set} in $X$ provided that for each open cover $\mathcal{U}$ of $X$,
 there is a map $f : X \to X$ such that $f$ is $\mathcal{U}$-close to the identity map on $X$ and the (closure of )image $f(X)$ misses $A$.
These concepts play important roles in infinite-dimensional topology.
A \textit{$Z$-embedding} means an embedding whose image is a $Z$-set.
We say that a space $X$ is \textit{strongly universal} for a class $\mathcal{C}$ if the following condition holds:
\begin{itemize}
 \item For each space $A \in \mathcal{C}$, each closed subset $B$ of $A$, each open cover $\mathcal{U}$ of $X$, and each map $f : A \to X$ such that the restriction $f|_B$ is a $Z$-embedding, there exists a $Z$-embedding $g : A \to X$ such that $g$ is $\mathcal{U}$-close to $f$ and $g|_B = f|_B$.
\end{itemize}
A space $X$ has \textit{the $\kappa$-discrete $n$-cells property}, $n < \omega$, if the following condition is satisfied:
\begin{itemize}
 \item For every open cover $\mathcal{U}$ of $X$ and every map $f : \bigoplus_{\gamma < \kappa} A_\gamma \to X$,
  where each $A_\gamma = \I^n$,
  there exists a map $g : \bigoplus_{\gamma < \kappa} A_\gamma \to X$ such that $g$ is $\mathcal{U}$-close to $f$ and the family $\{g(A_\gamma)\}_{\gamma < \kappa}$ is discrete in $X$.
\end{itemize}
The author \cite{Kos1} gave the following characterization to an $\ell_2^f(\kappa)$-manifold:

\begin{thm}\label{DCP-char.}
A connected space $X$ is an $\ell_2^f(\kappa)$-manifold if and only if the following conditions are satisfied:
\begin{enumerate}
 \item $X$ is a strongly countable-dimensional, $\sigma$-locally compact ANR of density $\kappa$;
 \item $X$ is strongly universal for the class of finite-dimensional compact metrizable spaces;
 \item every finite-dimensional compact subset of $X$ is a strong $Z$-set in $X$;
 \item $X$ has the $\kappa$-discrete $n$-cells property for every $n < \omega$.
\end{enumerate}
\end{thm}

In the above theorem, replacing ``ANR'' with ``AR'', we have a characterization of the model space $\ell_2^f(\kappa)$.
Given a simplicial complex $K$, we denote the polyhedron\footnote{In this paper, polyhedra are not needed to be metrizable.} of $K$ by $|K|$ and the $n$-skeleton of $K$ by $K^{(n)}$ for each $n < \omega$.
Regarding $\sigma \in K$ as a simplicial complex consisting of its faces, we write $\sigma^{(n)}$ as the union of $i$-faces of $\sigma$, $i \leq n$.
The next two lemmas are used in the proof of Theorem~E of \cite{Cu1}.

\begin{lem}\label{subd.1}
Let $Y = (Y,\rho)$ be a metric space, $K$ a simplicial complex and $f : |K| \to Y$ a map.
For each map $\alpha : Y \to (0,\infty)$, there exists a subdivision $K'$ of $K$ such that $\diam_\rho{f(\sigma)} < \inf_{x \in \sigma} \alpha f(x)$ for all $\sigma \in K'$.
\end{lem}

\begin{proof}
Let $\alpha : Y \to (0,\infty)$ be a map.
By the continuity of $\alpha$, we can find $0 < \delta(y) \leq \alpha(y)/4$ for each $y \in Y$ such that for every $y' \in B_\rho(y,\delta(y))$, $\alpha(y') \geq \alpha(y)/2$.
It follows from \cite[Theorem~4.7.11]{Sa6}, there is a subdivision $K'$ of $K$ that refines the open cover $\{f^{-1}(B_\rho(y,\delta(y))) \mid y \in Y\}$.
To show that $K'$ is the desired subdivision, take any $\sigma \in K'$.
Then we can choose $y \in Y$ so that $f(\sigma) \subset B_\rho(y,\delta(y))$,
 and hence $\diam_\rho f(\sigma) < 2\delta(y) \leq \alpha(y)/2$.
Observe that for each $x \in \sigma$, $\alpha f(x) \geq \alpha(y)/2$,
 which implies that $\diam_\rho{f(\sigma)} < \inf_{x \in \sigma} \alpha f(x)$.
Thus the proof is complete.
\end{proof}

\begin{lem}\label{subd.2}
For each map $\alpha : |K| \to (0,\infty)$ of the polyhedron of a simplicial complex $K$ and $\beta > 1$, there is a subdivision $K'$ of $K$ such that $\sup_{x \in \sigma} \alpha(x) < \beta\inf_{x \in \sigma} \alpha(x)$ for any $\sigma \in K'$.
\end{lem}

\begin{proof}
For each $x \in |K|$, we can choose an open neighborhood $U(x)$ of $x$ in $|K|$ so that if $y \in U(x)$,
 then $|\alpha(x) - \alpha(y)| < (\beta - 1)\alpha(x)/(\beta + 1)$.
Then $\mathcal{U} = \{U(x) \mid x \in |K|\}$ is an open cover of $|K|$.
According to Theorem~4.7.11 of \cite{Sa6}, there is a subdivision $K'$ of $K$ that refines $\mathcal{U}$.
Take any simplex $\sigma \in K'$ and any point $y \in \sigma$.
By the compactness of $\sigma$, we can find $z \in \sigma$ such that $\alpha(z) = \inf_{z' \in \sigma} \alpha(z')$.
Since $K'$ refines $\mathcal{U}$,
 there exists a point $x \in |K|$ such that $\sigma \subset U(x)$.
Then $|\alpha(x) - \alpha(y)| < (\beta - 1)\alpha(x)/(\beta + 1)$ and $|\alpha(x) - \alpha(z)| < (\beta - 1)\alpha(x)/(\beta + 1)$.
Observe that $\alpha(y) < 2\beta\alpha(x)/(\beta + 1) < \beta\alpha(z)$,
 which implies that $\sup_{z' \in \sigma} \alpha(z') < \beta\inf_{z' \in \sigma} \alpha(z')$.
Hence $K'$ is the desired subdivision.
\end{proof}

\section{Basic properties of $\fin(X)$}

In this section, we list some basic properties of the hyperspace $\fin(X)$.
According to Proposition~5.3 of \cite{MSY}, we have the following:

\begin{prop}\label{Borel}
A space $X$ is strongly countable-dimensional, $\sigma$-locally compact and of density $\kappa$ if and only if so is $\fin(X)$.
\end{prop}

\begin{prop}\label{nbd.dens.}
Let $X$ be a space.
For each point $x \in X$, if every neighborhood of $\{x\}$ in $\fin(X)$ is of density $\kappa$,
 then so is any neighborhood of $x$ in $X$.
\end{prop}

\begin{proof}
Since $\fin(X)$ is of density $\kappa$,
 so is $X$ due to Proposition~\ref{Borel}.
Let $x \in X$ be a point such that any neighborhood of $\{x\}$ in $\fin(X)$ is of density $\kappa$.
Assume that there is a neighborhood $U$ of $x$ whose density $< \kappa$.
As is easily observed,
 $\fin(U)$ is a neighborhood of $\{x\}$ in $\fin(X)$.
Applying Proposition~\ref{Borel} again, we have that the density of $\fin(U)$ is less than $\kappa$,
 which is a contradiction.
Consequently, every neighborhood of $x$ is of density $\kappa$.
\end{proof}

Combining Lemmas~2.3 and 3.6 with the proof of Theorem~2.4 in \cite{CN} (cf.~\cite[Proposition~3.1]{Yag}), we can establish the following:

\begin{prop}\label{ar}
A space $X$ is locally path-connected (connected and locally path-connected) if and only if $\fin(X)$ is an ANR (AR).
\end{prop}

\section{The strong universality of $\fin(X)$}

This section is devoted to verifying the strong universality of $\fin(X)$ for the class of  finite-dimensional compact metrizable spaces.
The following lemma follows from Lemma~2.2 of \cite{CN}.

\begin{lem}\label{Peano}
Let $X$ be a space.
If $\mathcal{A} \subset \fin(X)$ is locally connected and compact,
 then so is the union $\bigcup \mathcal{A} \subset X$.
\end{lem}

Using Curtis and Nhu's result \cite{CN}, we shall show the strong universality of a non-separable hyperspace $\fin(X)$.

\begin{prop}\label{str.univ.}
Let $X$ be a non-degenerate, connected, locally path-connected space.
Then $\fin(X)$ is strongly universal for the class of finite-dimensional compact metrizable spaces.
\end{prop}

\begin{proof}
Let $A$ be a finite-dimensional compact metrizable space, $B$ a closed set in $A$, $f : A \to \fin(X)$ a map such that $f|_B$ is an $Z$-embedding.
We show that for any open cover $\mathcal{U}$ of $\fin(X)$, there is a $Z$-embedding $g : A \to \fin(X)$ such that $g$ is $\mathcal{U}$-close to $f$ and $g|_B = f|_B$.
Since $A$ is finite-dimensional and compact,
 we can regard it as a closed subset of $\I^n$ for some $n < \omega$.
The space $X$ is connected and locally path-connected,
 and hence $\fin(X)$ is an AR by Proposition~\ref{ar}.
Therefore the map $f$ can be extended to a map $\tilde{f} : \I^n \to \fin(X)$.
Note that $\mathcal{A} = \tilde{f}(\I^n)$ is connected, locally connected and compact.
According to Lemma~\ref{Peano}, the union $\bigcup \mathcal{A} \subset X$ is also locally connected and compact,
 and hence it is locally path-connected and has finitely many components.
Note that each component of $\bigcup \mathcal{A}$ is open and closed.
We may assume that at least one component of $\bigcup \mathcal{A}$ is non-degenerate because $X$ has no isolated points and is locally path-connected.
Let
 $$\overline{\mathcal{A}} = \Big\{F \in \fin(\bigcup \mathcal{A}) \Bigm| F \text{ contains some element of } \mathcal{A}\Big\}.$$
As is easily observed,
 the space $\overline{\mathcal{A}}$ is separable and for each $F \in \fin(\bigcup \mathcal{A})$, $F \in \overline{\mathcal{A}}$ if $F$ contains some element of $\overline{\mathcal{A}}$.

We show that any $F \in \overline{\mathcal{A}}$ meets each component of $\bigcup \mathcal{A}$.
It is sufficient to prove that every $F \in \mathcal{A}$ intersects each component of $\bigcup \mathcal{A}$.
Suppose the contrary,
 so there are $F \in \mathcal{A}$ and a component $C$ of $\bigcup \mathcal{A}$ such that $F \cap C = \emptyset$.
Then $C$ and $\bigcup \mathcal{A} \setminus C$ are open and $F \subset \bigcup \mathcal{A} \setminus C$.
Observe that $\{F' \in \mathcal{A} \mid F' \cap C \neq \emptyset\}$ and $\{F' \in \mathcal{A} \mid F' \subset \bigcup \mathcal{A} \setminus C\}$ are non-empty disjoint open sets in $\mathcal{A}$.
Moreover,
 $$\mathcal{A} = \Big\{F' \in \mathcal{A} \Bigm| F' \cap C \neq \emptyset\Big\} \cup \Big\{F' \in \mathcal{A} \Bigm| F' \subset \bigcup \mathcal{A} \setminus C\Big\}.$$
This contradicts to the connectedness of $\mathcal{A}$.
Hence any $F \in \overline{\mathcal{A}}$ meets all components of $\bigcup \mathcal{A}$.
Applying Lemma~4.6 of \cite{CN}, we have that $\overline{\mathcal{A}}$ is strongly universal for the class of finite-dimensional compact metrizable spaces.
Therefore there exists an embedding $g : A \to \overline{\mathcal{A}} \subset \fin(X)$ such that $g$ is $\mathcal{U}$-close to $f$ and $g|_B = f|_B$.
By Lemma~3.8 of \cite{CN}, the compact image $g(A)$ is a $Z$-set in $\fin(X)$.
The proof is completed.
\end{proof}

\section{The strong $Z$-set property of compact sets in $\fin(X)$}

In this section, we will discuss the strong $Z$-set property of compact subsets of $\fin(X)$.

\begin{lem}\label{sph.}
If a metric space $X = (X,d)$ is non-degenerate and connected,
 then for each $x \in X$ and $0 < \epsilon < \diam_d{X}/2$, there exists a point $y \in X$ such that $d(x,y) = \epsilon$.
\end{lem}

\begin{proof}
Suppose the contrary.
Then $X$ can be separated by disjoint non-empty open subsets $B_d(x,\epsilon)$ and $X \setminus \overline{B_d}(x,\epsilon)$,
 which contradicts to the connectedness of $X$.
The proof is complete.
\end{proof}

Let $\cpt(X)$ be the hyperspace of non-empty compact sets in a space $X$ with the Vietoris topology.
Note that $\fin(X) \subset \cpt(X)$.
The analogues of the following two lemmas for $\cpt(X)$ are used in the proof of Theorem~E of \cite{Cu1}.

\begin{lem}\label{subseq.}
Let $X = (X,d)$ be a metric space.
Suppose that $\{A_n\}_{n < \omega}$ is a sequence in $\fin(X) = (\fin(X),d_H)$ converging to $A \in \fin(X)$.
Then for each $B_n \subset A_n$, $\{B_n\}_{n < \omega}$ has a subsequence converging to some $B \subset A$.
\end{lem}

\begin{proof}
According to Lemma~1.11.2.~(3)\footnote{This holds without the assumption that $X$ is separable.} of \cite{Mil3}, $\tilde{A} = A \cup \bigcup_{n < \omega} A_n$ is compact.
Hence the hyperspace $\cpt(\tilde{A}) = (\cpt(\tilde{A}),(d|_{\tilde{A} \times \tilde{A}})_H)$ is compact, see \cite[Theorem~5.12.5.~(3)]{Sa6},
 which implies that $\{B_n\}_{n < \omega}$ has a subsequence $\{B_{n_i}\}_{i < \omega}$ converging to some $B \in \cpt(\tilde{A})$.
By Lemma~1.11.2.~(2)\footnotemark[4] of \cite{Mil3}, we have
\begin{align*}
 B &= \{x \in X \mid \text{for each } i < \omega, \text{ there is } b_{n_i} \in B_{n_i} \text{ such that } \lim_{i \to \infty} b_{n_i} = x\}\\
 &\subset \{x \in X \mid \text{for each } i < \omega, \text{ there is } a_{n_i} \in A_{n_i} \text{ such that } \lim_{i \to \infty} a_{n_i} = x\} = A.
\end{align*}
Thus the proof is complete.
\end{proof}

\begin{lem}\label{arc}
Let $X = (X,d)$ be a metric space and $\alpha : \fin(X) \to (0,\infty)$ be a map.
If $X$ is locally path-connected,
 then there exists a map $\beta : \fin(X) \to (0,\infty)$ such that for any $A \in \fin(X)$, each point $x \in \overline{N_d}(A,\beta(A))$ has an arc $\gamma : \I \to X$ from some point of $A$ to $x$ of $\diam_d{\gamma(\I)} < \alpha(A)$.
\end{lem}

\begin{proof}
For each $A \in \fin(X)$, let
 $$\Xi(A) = \left\{
  \begin{array}{l|l}
  \eta > 0 &\left.
  \begin{array}{ll}
  \text{there exists } 0 < \epsilon < \alpha(A) \text{ such that for any } a \in A \text{ and } x \in \overline{B_d}(a,\eta),\\
  \text{there is an arc from } a \text{ to } x \text{ of diameter } < \epsilon
  \end{array}
  \right.
  \end{array}\right\}$$
 and $\xi(A) = \sup\Xi(A)$.
Note that $\Xi(A) \neq \emptyset$ for all $A \in \fin(X)$.
Indeed, let $0 < \epsilon < \alpha(A)$.
Since $X$ is locally path-connected,
 and hence locally arcwise-connected \cite[Corollary~5.14.7]{Sa6},
 for each $a \in A$, there exists $\eta(a) > 0$ such that for any $x \in \overline{B_d}(a,\eta(a))$, $a$ and $x$ are connected by an arc of diameter $< \epsilon$.
Then $\eta = \min_{a \in A}\eta(a) \in \Xi(A)$.
By the definition, $\xi(A) \leq \alpha(A)$. 

We shall show that $\xi$ is lower semi-continuous.
Take any $t \in (0,\infty)$ and any $A \in \xi^{-1}((t,\infty))$.
Then we can choose $t < \eta \leq \xi(A)$ so that there is $0 < \epsilon < \alpha(A)$ such that for any $a \in A$ and any $x \in \overline{B_d}(a,\eta)$, $a$ and $x$ are connected by an arc of diameter $< \epsilon$.
Since $X$ is locally arcwise-connected,
 there exists $\delta_1 > 0$ such that any $a \in A$ and any $x \in \overline{B_d}(a,\delta_1)$ are connected by an arc of diameter $< (\alpha(A) - \epsilon)/2$.
By the continuity of $\alpha$, we can find $\delta_2 > 0$ such that for each $B \in B_{d_H}(A,\delta_2)$, $|\alpha(A) - \alpha(B)| < (\alpha(A) - \epsilon)/2$.
Let $\delta = \min\{\delta_1, \delta_2, (\eta - t)/2\}$ and $B \in B_{d_H}(A,\delta)$.
Observe that $(\alpha(A) + \epsilon)/2 < \alpha(B)$.
Fix any $b \in B$ and any $x \in \overline{B_d}(b,(\eta + t)/2)$.
Since $d_H(A,B) < \delta$,
 we can take $a \in A$ such that $d(a,b) < \delta \leq \delta_1$,
 and hence there exists an arc $\gamma_1$ from $b$ to $a$ of diameter $< (\alpha(A) - \epsilon)/2$.
On the other hand,
 $$d(a,x) \leq d(a,b) + d(b,x) < \delta + (\eta + t)/2 \leq (\eta - t)/2 + (\eta + t)/2 = \eta,$$
 which implies that there is an arc $\gamma_2$ from $a$ to $x$ of diameter $< \epsilon$.
Joining these arcs $\gamma_1$ and $\gamma_2$, we can obtain an arc from $b$ to $x$ of diameter $< (\alpha(A) - \epsilon)/2 + \epsilon$,
 that is less than $\alpha(B)$.
Hence $t < (\eta + t)/2 \leq \xi(B)$,
 which means that $\xi$ is lower semi-continuous.

According to Theorem~2.7.6 of \cite{Sa6}, we can find a map $\beta : \fin(X) \to (0,\infty)$ such that $0 < \beta(A) < \xi(A)$ for all $A \in \fin(X)$,
 that is the desired map.
\end{proof}

The next lemma is useful to detect a strong $Z$-set in an ANR.

\begin{lem}[Lemma~7.2 of \cite{Cu4}]\label{str.Z}
Let $A$ be a topologically complete, closed subset of an ANR $Y$.
If $A$ is a countable union of strong $Z$-sets in $Y$,
 then it is a strong $Z$-set.
\end{lem}

We denote the cardinality of a set $A$ by $\card{A}$.
For each $k < \omega$, let $\fin^k(X) = \{A \in \fin(X) \mid \card{A} \leq k\}$.
As is easily observed, $\fin^k(X)$ is closed in $\fin(X)$.

\begin{prop}\label{str.Z_sigma}
Suppose that $X$ is non-degenerate, connected and locally path-connected.
Then for each $k < \omega$, $\fin^k(X)$ is a strong $Z$-set in $\fin(X)$.
\end{prop}

\begin{proof}
Let $\mathcal{U}$ be an open cover of $\fin(X)$ and $k < \omega$.
We shall construct a map $\phi : \fin(X) \to \fin(X)$ so that $\phi$ is $\mathcal{U}$-close to the identity map on $\fin(X)$ and $\cl{\phi(\fin(X))} \cap \fin^k(X) = \emptyset$,
 where for a subset $\mathcal{A} \subset \fin(X)$, $\cl{\mathcal{A}}$ means the closure of $\mathcal{A}$ in $\fin(X)$.
Take an open cover $\mathcal{V}$ of $\fin(X)$ that is a star-refinement of $\mathcal{U}$.
Since $\fin(X)$ is an AR by Proposition~\ref{ar},
 there are a simplicial complex $K$ and maps $f : \fin(X) \to |K|$, $g : |K| \to \fin(X)$ such that $gf$ is $\mathcal{V}$-close to the identity map on $\fin(X)$, refer to \cite[Theorem~6.6.2]{Sa6}.
It remains to show that there exists a map $h : |K| \to \fin(X)$ $\mathcal{V}$-close to $g$ such that $\cl{h(|K|)} \cap \fin^k(X) = \emptyset$ because $\phi = hf$ will be the desired map.

Fix any admissible metric $d$ on $X$.
Then we can take a map $\alpha : \fin(X) \to (0,\min\{1,\diam_d{X}\})$ so that the collection $\{B_{d_H}(A,2\alpha(A)) \mid A \in \fin(X)\}$ refines $\mathcal{V}$.
Since $X$ is locally path-connected,
 due to Lemma~\ref{arc}, there is a map $\beta : \fin(X) \to (0,\infty)$ such that for any $A \in \fin(X)$, each point $x \in \overline{N_d}(A,\beta(A))$ has an arc $\gamma : \I \to X$ from some point of $A$ to $x$ of $\diam_d{\gamma(\I)} < \alpha(A)/2$.
We may assume that $\beta(A) \leq \alpha(A)/2$ for every $A \in \fin(X)$.
Combining Lemmas~\ref{subd.1} with \ref{subd.2}, we can replace $K$ with a subdivision so that for each $\sigma \in K$,
\begin{enumerate}
 \item $\diam_{d_H}{g(\sigma)} < \inf_{y \in \sigma} \beta g(y)/2$,
 \item $\sup_{y \in \sigma} \beta g(y) < 2\inf_{y \in \sigma} \beta g(y)$,
 \item $\sup_{y \in \sigma} \alpha g(y) < 4\inf_{y \in \sigma} \alpha g(y)/3$.
\end{enumerate}

For every $v \in K^{(0)}$, fix a point $x(v) \in g(v)$.
According to Lemma~\ref{sph.}, we can find a point $z(v,j) \in X$ with $d(x(v),z(v,j)) = j\beta g(v)/(4(k+1))$ for each $j = 0, \cdots, k$.
Let $h(v) = g(v) \cup \{z(v,j) \mid j = 0, \cdots, k\}$.
Clearly, $\card{h(v)} \geq k + 1$ and $d_H(g(v),h(v)) \leq \beta g(v) \leq \alpha g(v)/2$.
Observe that for any $0 \leq i < j \leq k$,
\begin{align*}
 d(z(v,i),z(v,j)) &\geq |d(x(v),z(v,i)) - d(x(v),z(v,j))| = (j-i)\beta g(v)/(4(k+1))\\
 &\geq \beta g(v)/(4(k+1)).
\end{align*}

Next, we will extend $h$ over $|K^{(1)}|$.
Let $\sigma \in K^{(1)} \setminus K^{(0)}$, $\sigma^{(0)} = \{v_1, v_2\}$ and $\hat\sigma$ be the barycenter of $\sigma$.
Due to conditions (1) and (2), we have for any $m = 1, 2$ and $j = 0, \cdots, k$,
\begin{align*}
 d(z(v_m,j),g(\hat\sigma)) &\leq d(z(v_m,j),g(v_m)) + d_H(g(v_m),g(\hat\sigma)) < \beta g(v_m)/4 + \diam_{d_H}{g(\sigma)}\\
 &< \sup_{y \in \sigma} \beta g(y)/4 + \inf_{y \in \sigma} \beta g(y)/2 < \inf_{y \in \sigma} \beta g(y) \leq \beta g(\hat\sigma).
\end{align*}
Hence there is an arc $\gamma(m,j) : \I \to X$ from some point of $g(\hat\sigma)$ to $z(v_m,j)$ of $\diam_d{\gamma(m,j)(\I)} < \alpha g(\hat\sigma)/2$ by Lemma~\ref{arc}.
Define
 $$h(\hat\sigma) = g(\hat\sigma) \cup \{z(v_m,j) \mid m = 1, 2 \text{ and } j = 0, \cdots, k\}.$$
Note that $\card{h(\hat\sigma)} \geq k + 1$.
Moreover, $d_H(g(\hat\sigma),h(\hat\sigma)) \leq \beta g(\hat\sigma) \leq \alpha g(\hat\sigma)/2$.
Let $\phi : \I \to \fin(X)$ be a map defined by
 $$\phi(t) = g(\hat\sigma) \cup \{\gamma(m,j)(t) \mid m = 1, 2 \text{ and } j = 0, \cdots, k\},$$
 which is a path from $g(\hat\sigma)$ to $h(\hat\sigma)$.
For each $m = 1, 2$, define a map $h : \langle v_m,\hat\sigma \rangle \to \fin(X)$ of the segment between $v_m$ and $\hat\sigma$ in $\sigma$ as follows:
 $$h((1 - t)v_m + t\hat\sigma) = \left\{
 \begin{array}{ll}
  g((1 - 2t)v_m + 2t\hat\sigma) \cup \{z(v_m,j) \mid j = 0, \cdots, k\} &\text{if } 0 \leq t \leq 1/2,\\
  \phi(2t - 1) \cup \{z(v_m,j) \mid j = 0, \cdots, k\} &\text{if } 1/2 \leq t \leq 1.
 \end{array}
 \right.$$
Then for every $y \in \sigma$, when $y = (1 - t)v_m + t\hat\sigma$, $0 \leq t \leq 1/2$,
\begin{align*}
 d_H(g(\hat\sigma),h(y)) &\leq \max\{d_H(g(\hat\sigma),g((1 - 2t)v_m + 2t\hat\sigma)),\max\{d(z(v_m,j),g(\hat\sigma)) \mid j = 0, \cdots, k\}\}\\
 &\leq \max\{\diam_{d_H}{g(\sigma)},\beta g(\hat\sigma)\} \leq \max\{\inf_{y' \in \sigma} \beta g(y')/2,\beta g(\hat\sigma)\} \leq \beta g(\hat\sigma) \leq \alpha g(\hat\sigma)/2,
\end{align*}
 and when $y = (1 - t)v_m + t\hat\sigma$, $1/2 \leq t \leq 1$,
\begin{align*}
 d_H(g(\hat\sigma),h(y)) &\leq \max\{d_H(g(\hat\sigma),\phi(2t - 1)),\max\{d(z(v_m,j),g(\hat\sigma)) \mid j = 0, \cdots, k\}\}\\
 &\leq \max\{\diam_d{\gamma(n,j)(\I)}, \beta g(\hat\sigma) \mid n = 1, 2 \text{ and } j = 0, \cdots, k\}\\
 &\leq \max\{\alpha g(\hat\sigma)/2,\beta g(\hat\sigma)\} = \alpha g(\hat\sigma)/2.
\end{align*}
Hence, due to condition (3), we have
\begin{align*}
 d_H(g(y),h(y)) &\leq d_H(g(y),g(\hat\sigma)) + d_H(g(\hat\sigma),h(y)) \leq \diam_{d_H}{g(\sigma)} + \alpha g(\hat\sigma)/2\\
 &< \inf_{y' \in \sigma} \beta g(y')/2 + \alpha g(\hat\sigma)/2 \leq \beta g(\hat\sigma)/2 + \alpha g(\hat\sigma)/2 \leq 3\alpha g(\hat\sigma)/4\\
 &\leq 3\sup_{y' \in \sigma} \alpha g(y')/4 < \inf_{y' \in \sigma} \alpha g(y') \leq \alpha g(y).
\end{align*}
Note that for each $y \in \sigma$, $h(y)$ contains $\{z(v_1,j) \mid j = 0, \cdots, k\}$ or $\{z(v_2,j) \mid j = 0, \cdots, k\}$,
 so $\card{h(y)} \geq k + 1$.

By induction, we shall construct a map $h : |K| \to \fin(X)$ such that for each $y \in \sigma \in K \setminus K^{(0)}$, $h(y) = \bigcup_{a \in A(y)} h(a)$ for some $A(y) \in \fin(|\sigma^{(1)}|)$.
Assume that $h$ extends over $|K^{(n)}|$ for some $n < \omega$ such that for every $y \in \sigma \in K^{(n)} \setminus K^{(0)}$, $h(y) = \bigcup_{a \in A(y)} h(a)$ for some $A(y) \in \fin(|\sigma^{(1)}|)$.
Take any $\sigma \in K^{(n + 1)} \setminus K^{(n)}$.
By Lemma~3.3 of \cite{CN}, there exists a map $r : \sigma \to \fin(\partial{\sigma})$ such that $r(y) = \{y\}$ for all $y \in \partial{\sigma}$,
 where $\partial{\sigma}$ is the boundary of $\sigma$.
The map $h|_{\partial{\sigma}}$ induces $\tilde{h} : \fin(\partial{\sigma}) \to \fin(X)$ defined by $\tilde{h}(A) = \bigcup_{a \in A} h(a)$.
Then we can obtain the composition $h_\sigma = \tilde{h}r : \sigma \to \fin(X)$.
It follows from the definition that $h_\sigma|_{\partial{\sigma}} = h|_{\partial{\sigma}}$.
Observe that for each $y \in \sigma$,
 $$h_\sigma(y) = \tilde{h}r(y) = \bigcup_{y' \in r(y)} h(y') = \bigcup_{y' \in r(y)} \bigcup_{a \in A(y')} h(a) = \bigcup_{a \in \bigcup_{y' \in r(y)} A(y')} h(a),$$
 where $h(y') = \bigcup_{a \in A(y')} h(a)$ for some $A(y') \in \fin(|\sigma^{(1)}|)$ by the inductive assumption.
Thus we can extend $h$ over $|K^{(n + 1)}|$ by $h|_\sigma = h_\sigma$ for all $\sigma \in K^{(n + 1)} \setminus K^{(n)}$.

After completing this induction, we can obtain a map $h : |K| \to \fin(X)$.
For each $\sigma \in K \setminus K^{(0)}$, each $y \in \sigma$ and each $a \in |\sigma^{(1)}|$, we get
\begin{align*}
 d_H(g(y),h(a)) &\leq d_H(g(y),g(a)) + d_H(g(a),h(a)) < \diam_{d_H}{g(\sigma)} + \alpha g(a)\\
 &< \inf_{y' \in \sigma} \beta g(y')/2 + \sup_{y' \in \sigma} \alpha g(y') \leq \inf_{y' \in \sigma} \alpha g(y')/4 + 4\inf_{y' \in \sigma} \alpha g(y')/3\\
 &= 19\inf_{y' \in \sigma} \alpha g(y')/12 < 2\alpha g(y).
\end{align*}
Therefore we have
 $$d_H(g(y),h(y)) = d_H\Bigg(g(y),\bigcup_{a \in \bigcup_{y' \in r(y)} A(y')} h(a)\Bigg) \leq \max_{a \in \bigcup_{y' \in r(y)} A(y')} d_H(g(y),h(a)) < 2\alpha g(y),$$
 which implies that $h$ is $\mathcal{V}$-close to $g$.
Remark that $\{z(v,j) \mid j = 0, \cdots, k\} \subset h(y)$ for some $v \in \sigma^{(0)}$,
 and hence $\card{h(y)} \geq k + 1$.
It follows that $h(|K|) \cap \fin^k(X) = \emptyset$.
Then we may replace $h(y)$ with $g(y) \cup h(y)$ for every $y \in |K|$,
 so we have $g(y) \subset h(y)$.
The rest of this proof is to show that $\cl{h(|K|)} \cap \fin^k(X) = \emptyset$.

Suppose that there exists a sequence $\{y_n\}_{n < \omega}$ of $|K|$ such that $\{h(y_n)\}_{n < \omega}$ converges to some $A \in \fin^k(X)$.
Take the carrier $\sigma_n \in K$ of $y_n$ and choose $v_n \in \sigma_n^{(0)}$ so that $\{z(v_n,j) \mid j = 0, \cdots, k\} \subset h(y_n)$.
Since $g(y_n) \subset h(y_n)$,
 replacing $\{g(y_n)\}_{n < \omega}$ with a subsequence, we can obtain $B \subset A$ to which $\{g(y_n)\}_{n < \omega}$ converges by Lemma~\ref{subseq.}.
Then $\{\beta g(y_n)\}_{n < \omega}$ converges to $\beta(B) > 0$.
On the other hand, for every $\epsilon > 0$, there exists $n_0 < \omega$ such that if $n \geq n_0$,
 then $d_H(h(y_n),A) < \epsilon$.
Then we can choose $0 \leq i(n) < j(n) \leq k$ for each $n \geq n_0$ so that $z(v_n,i(n)), z(v_n,j(n)) \in B_d(a,\epsilon)$ for some $a \in A$ because
 $$\card{A} \leq k < k + 1 = \card\{z(v_n,j) \mid j = 0, \cdots, k\}.$$
Note that
\begin{align*}
 \beta g(y_n)/(8(k+1)) &\leq \sup_{y \in \sigma_n} \beta g(y)/(8(k+1)) < \inf_{y \in \sigma_n} \beta g(y)/(4(k+1)) \leq \beta g(v_n)/(4(k+1))\\
 &\leq d(z(v_n,i(n)), z(v_n,j(n))) < 2\epsilon,
\end{align*}
 which means that $\{\beta g(y_n)\}_{n < \omega}$ converges to $0$.
This is a contradiction.
Consequently, $\cl{h(|K|)} \cap \fin^k(X) = \emptyset$.
\end{proof}

Combining Lemma~\ref{str.Z} with Proposition~\ref{str.Z_sigma}, we can obtain the following:

\begin{prop}\label{cpt.str.Z}
Let $X$ be a non-degenerate, connected and locally path-connected space.
Then every compact subset of $\fin(X)$ is a strong $Z$-set.
\end{prop}

\section{The $\kappa$-discrete $n$-cells property of $\fin(X)$}

This section is devoted to detecing the $\kappa$-discrete $n$-cells property in $\fin(X)$.
First, we show the following lemma:

\begin{lem}\label{lfap}
Let $X$ be a locally path-connected space.
Suppose that any neighborhood of each point in $X$ contains a discrete subset of cardinality $\geq \kappa$.
Then the hyperspace $\fin(X)$ satisfies the following:
\begin{itemize}
 \item Let $K_\gamma$ be a simplicial complex, $\gamma < \kappa$.
 For each open cover $\mathcal{V}$ of $\fin(X)$, and each map $g : \bigoplus_{\gamma < \kappa} |K_\gamma| \to \fin(X)$, there exists a map $h : \bigoplus_{\gamma < \kappa} |K_\gamma| \to \fin(X)$ such that $h$ is $\mathcal{V}$-close to $g$ and the family $\{h(|K_\gamma|)\}_{\gamma < \kappa}$ is locally finite in $\fin(X)$.
\end{itemize}
\end{lem}

\begin{proof}
Fix any admissible metric $d$ on $X$.
Let $\alpha, \beta : \fin(X) \to (0,1)$ be the same maps and replace each $K_\gamma$, $\gamma < \kappa$, with the same subdivision as in Proposition~\ref{str.Z_sigma}.
For each $n < \omega$, we can find a locally finite open cover $\mathcal{V}_n$ of $X$ of mesh $< 1/n$.
By the assumption, for every $n < \omega$, each non-empty $V \in \mathcal{V}_n$ contains a discrete subset $Z(V) = \{z_V^\gamma\}_{\gamma < \kappa}$ of cardinality $\kappa$.
Let $Z^\gamma(n) = \{z_V^\gamma \mid \emptyset \neq V \in \mathcal{V}_n\}$, $\gamma < \kappa$.
Here we may assume that $Z^\gamma(n) \cap Z^\tau(n) = \emptyset$ if $\gamma \neq \tau$.
Indeed, we shall show it by transfinite induction.
Suppose that for some $\gamma < \kappa$, $Z^\tau(n) \cap Z^{\tau'}(n) = \emptyset$ if $\tau < \tau' < \gamma$.
By the local finiteness of $\mathcal{V}_n$, for every $z_V^\gamma \in Z^\gamma(n)$, $\{V' \in \mathcal{V}_n \mid z_V^\gamma \in V'\}$ is finite.
Since $X$ has no isolated points and each $Z(V)$ is discrete,
 we can find a point $x_V^\gamma \in V$ sufficiently close to $z_V^\gamma$ such that $x_V^\gamma \notin \bigcup_{\tau < \gamma} Z^\tau(n)$ and even if $z_V^\gamma$ is substituted by $x_V^\gamma$,
 $Z(V)$ is still discrete.
Due to this substitution, we have that $Z^\tau(n) \cap Z^\gamma(n) = \emptyset$ for all $\tau < \gamma$.
Put $Z(n) = \bigcup_{\emptyset \neq V \in \mathcal{V}_n} Z(V) = \bigoplus_{\gamma < \kappa} Z^\gamma(n)$,
 that is locally finite in $X$.

To begin with, we shall construct the restriction $h|_{K_\gamma^{(0)}}$, $\gamma < \kappa$.
For every $v \in K_\gamma^{(0)}$, there is $n_v \geq 2$ such that $1/n_v < \beta g(v)/4 \leq 1/(n_v - 1)$.
Then we can find a point $z^\gamma(v) \in Z^\gamma(n_v) \subset Z(n_v)$, $\gamma < \kappa$, so that $d(z^\gamma(v),g(v)) < 1/n_v$.
Remark that for any $u \in K_\tau^{(0)}$ and $u' \in K_{\tau'}^{(0)}$ with $n_{u} = n_{u'}$, $z^\tau(u) \neq z^{\tau'}(u')$ if $\tau < \tau' < \kappa$.
Let $h(v) = g(v) \cup \{z^\gamma(v)\} \in \fin(X)$.
Note that
 $$d_H(g(v),h(v)) < 1/n_v < \beta g(v)/4 \leq \alpha g(v)/2.$$

After the same construction of map as in Proposition~\ref{str.Z_sigma}, we can obtain a map $h : \bigoplus_{\gamma < \kappa} |K_\gamma| \to \fin(X)$ so that $h$ is $\mathcal{V}$-close to $g$ and for any $\gamma < \kappa$ and any $y \in |K_\gamma|$, $h(y)$ contains a point $z^\gamma(v) \in Z(n_v)$ for some $v \in \sigma^{(0)}$,
 where $\sigma \in K_\gamma$ is the carrier of $y$.
Here we may replace $h(y)$ with the union $g(y) \cup h(y)$ for every $y \in |K_\gamma|$,
 so we have $g(y) \subset h(y)$.
It remains to show that $\{h(|K_\gamma|)\}_{\gamma < \kappa}$ is locally finite in $\fin(X)$.

Suppose the contrary,
 so we can find a sequence $\{\gamma_i\}_{i < \omega}$ of $\kappa$ such that $\gamma_i \neq \gamma_j$ if $i \neq j$,
 and $\{h(y_{\gamma_i})\}_{i < \omega}$, $y_{\gamma_i} \in |K_{\gamma_i}|$, converges to some $A \in \fin(X)$.
For the sake of convenience, replace each $\gamma_i$ with $i$.
Let $\sigma_i \in K_i$ be the carrier of $y_i$ and choose a vertex $v_i \in \sigma_i^{(0)}$ so that $z^i(v_i) \in h(y_i)$.
Since $g(y_i) \subset h(y_i)$,
 replacing $\{g(y_i)\}_{i < \omega}$ with a subsequence, we can obtain a subset $B \subset A$ to which $\{g(y_i)\}_{i < \omega}$ converges by Lemma~\ref{subseq.}.
Then $\{\beta g(y_i)\}_{i < \omega}$ is converging to $\beta(B) > 0$.
On the other hand, any subsequence of $\{z^i(v_i)\}_{i < \omega}$ has an accumulation point in $A$ because $\{h(y_i)\}_{i < \omega}$ converges to $A$.
Observe that $\{n_{v_i}\}_{i < \omega}$ diverges to $\infty$.
Indeed, supposing the converse, we can find $n_0 < \omega$ and replace $\{n_{v_i}\}_{i < \omega}$ with a subsequence so that $n_{v_i} = n_0$ for all $i < \omega$.
By the choice of $z^i(v_i)$, $\{z^i(v_i)\}_{i < \omega}$ is pairwise distinct and contained in the locally finite subset $Z(n_0)$,
 which is a contradiction.
Hence $\{\beta g(v_i)\}_{i < \omega}$ converges to $0$ since $\beta g(v_i)/4 \leq 1/(n_{v_i} - 1)$.
Moreover, it follows from condition (2) of Proposition~\ref{str.Z_sigma} that for every $i < \omega$,
 $$\beta g(y_i) \leq \sup_{y \in \sigma_i} \beta g(y) < 2\inf_{y \in \sigma_i} \beta g(y) \leq 2\beta g(v_i).$$
Therefore $\{\beta g(y_i)\}_{i < \omega}$ converges to $0$,
 which is a contradiction.
Thus we conclude that the family $\{h(|K_\gamma|)\}_{\gamma < \kappa}$ is locally finite in $\fin(X)$.
\end{proof}

It is said that a space $X$ has \textit{the countable locally finite approximation property} provided that for every open cover $\mathcal{U}$ of $X$, there exists a sequence $\{f_n : X \to X\}_{n < \omega}$ of maps such that each $f_n$ is $\mathcal{U}$-close to the identity map on $X$ and the family $\{f_n(X)\}_{n < \omega}$ is locally finite in $X$.

\begin{prop}\label{clfap}
Let $X$ be a locally path-connected and nowhere locally compact space.
Then $\fin(X)$ has the countable locally finite approximation property.
\end{prop}

\begin{proof}
By the same argument as Proposition~\ref{str.Z_sigma}, we need only to show that for any simplicial complex $K$, any map $g : |K| \to \fin(X)$ and any open cover $\mathcal{V}$ of $\fin(X)$, there exists a map $g_i : |K| \to \fin(X)$ for each $i < \omega$ such that each $g_i$ is $\mathcal{V}$-close to $g$ and $\{g_i(|K|)\}_{i < \omega}$ is locally finite in $\fin(X)$.
Since $X$ is nowhere locally compact,
 it satisfies the assumption as in Lemma~\ref{lfap} with respect to $\kappa = \omega$.
Hence this lemma guarantees the countable locally finite approximation property of $\fin(X)$.
\end{proof}

The next two lemmas are useful for recognizing the $\kappa$-discrete $n$-cells property in a space.

\begin{lem}[Lemma~3.1 of \cite{BZ}]\label{loc.fin.}
Let $n < \omega$.
A space $X$ has the $\kappa$-discrete $n$-cells property if and only if the following condition holds:
\begin{itemize}
 \item For each open cover $\mathcal{U}$ of $X$, and each map $f : \bigoplus_{\gamma < \kappa} A_\gamma \to X$, where each $A_\gamma = \I^n$,
 there is a map $g : \bigoplus_{\gamma < \kappa} A_\gamma \to X$ such that $g$ is $\mathcal{U}$-close to $f$ and the family $\{g(A_\gamma)\}_{\gamma < \kappa}$ is locally finite in $X$.
\end{itemize}
\end{lem}

\begin{lem}[Lemma~3.2 of \cite{BZ}]\label{cof.}
Let $n < \omega$.
A space $X$ with the countable locally finite approximation property has the $\kappa$-discrete $n$-cells property if and only if $X$ has the $\lambda$-discrete $n$-cells property for any $\lambda \leq \kappa$ of uncountable cofinality.
\end{lem}

A subset $A$ of a metric space $X = (X,d)$ is called \textit{$\delta$-discrete}, $\delta > 0$, provided that for any distinct points $a, a' \in A$, $d(a,a') \geq \delta$.
The following lemma follows from the proof of \cite[Lemma~6.2]{BZ}:

\begin{lem}\label{discr.}
Let $X$ be a space and $\kappa$ be of uncountable cofinality.
For each subset $A \subset X$ of density $\geq \kappa$, it contains a discrete subset of cardinality $\geq \kappa$.
\end{lem}

\begin{proof}
Let $A \subset X$ be of density $\geq \kappa$.
For each $n < \omega$, take any maximal $2^{-n}$-discrete subset $D(n) \subset A$.
If some $D(n)$ is of cardinality $\geq \kappa$,
 then the proof is finished.
Conversely, suppose that every $D(n)$ is of cardinality $< \kappa$.
Then $\bigcup_{n < \omega} D(n)$ is dense in $A$ due to the maximality of $D(n)$.
On the other hand, since $\kappa$ is of uncountable cofinality,
 the cardinality of the countable union $\bigcup_{n < \omega} D(n)$ is less than $\kappa$,
 which contradicts to that $A$ is of density $\geq \kappa$.
Thus we conclude that $A$ contains a discrete subset of cardinality $\geq \kappa$.
\end{proof}

Now, we will prove the following:

\begin{prop}\label{dcp}
Let $X$ be a locally path-connected and nowhere locally compact space.
If any neighborhood of each point in $X$ is of density $\geq \kappa$,
 then the hyperspace $\fin(X)$ has the $\kappa$-discrete $n$-cells property for every $n < \omega$.
\end{prop}

\begin{proof}
The hyperspace $\fin(X)$ has the countable locally finite approximation property by Proposition~\ref{clfap}.
Hence we may assume that $\kappa$ is of uncountable cofinality due to Lemma~\ref{cof.}.
Let $n < \omega$ and take any open cover $\mathcal{V}$ of $\fin(X)$.
By virtue of Lemma~\ref{loc.fin.}, it suffices to prove that for any map $g : \bigoplus_{\gamma < \kappa} A_\gamma \to \fin(X)$, where each $A_\gamma = \I^n$,
 there exists a $\mathcal{V}$-close map $h : \bigoplus_{\gamma < \kappa} A_\gamma \to \fin(X)$ to $g$ such that $\{h(A_\gamma)\}_{\gamma < \kappa}$ is locally finite in $\fin(X)$.
It follows from Lemma~\ref{discr.} that any neighborhood of each point in $X$ contains a discrete subset of cardinality $\geq \kappa$.
Using Lemma~\ref{lfap}, we can obtain the desired map $h$.
Thus the proof is completed.
\end{proof}

\section{Proof of the main theorem}

In this final section, applying the characterization~2.1, we prove the main theorem.

\begin{proof}[Proof of the main theorem]
The separable case follows from Theorem~\ref{l2f},
 so we only consider $\kappa$ be uncountable.

(The ``only if'' part)~According to Propositions~\ref{Borel} and \ref{ar}, the hyperspace $\fin(X)$ is a strongly countable-dimensional and $\sigma$-locally compact AR of density $\kappa$.
It follows from Proposition~\ref{str.univ.} that $X$ is strongly universal for the class of finite-dimensional compact metrizable spaces.
Due to Proposition~\ref{cpt.str.Z}, every finite-dimensional compact subset of $\fin(X)$ is a strong $Z$-set in $\fin(X)$.
Moreover, the $\kappa$-discrete $n$-cells property of $\fin(X)$ for every $n < \omega$ follows from Proposition~\ref{dcp}.
Using Theorem~\ref{DCP-char.}, we conclude that $\fin(X)$ is homeomorphic to $\ell_2^f(\kappa)$.

(The ``if'' part)~Since $\fin(X)$ is homeomorphic to $\ell_2^f(\kappa)$,
 it is a strongly countable-dimensional and $\sigma$-locally compact AR,
 and hence the space $X$ is connected, locally path-connected, strongly countable-dimensional and $\sigma$-locally compact by Propositions~\ref{Borel} and \ref{ar}.
Remark that for each $A \in \fin(X)$, all neighborhoods of $A$ are of density $\kappa$.
It follows from Proposition~\ref{nbd.dens.} that any neighborhood of each point in $X$ is also of density $\kappa$.
The proof is complete.
\end{proof}

\end{document}